\documentclass[12pt]{article}

\usepackage[margin=1in]{geometry}
\usepackage{times}
\usepackage{amssymb}
\usepackage{amsmath}
\usepackage{graphicx}
\usepackage{amsthm}
\usepackage{enumitem}
\usepackage{hyperref}

\newtheorem{theorem}{Theorem}
\newtheorem{corollary}[theorem]{Corollary}
\newtheorem{proposition}[theorem]{Proposition}
\newtheorem{definition}[theorem]{Definition}
\newtheorem{remark}[theorem]{Remark}

\usepackage[numbers,sort&compress]{natbib}

\title{Multicriticality and Scaling:\\ Mellin Spectral Theory, and the\\
Decoupling of Geometric and Spectral Exponents}

\author{
Laurence A.\ Jacobs\textsuperscript{1,2,*},\quad
Alejandro Frank\textsuperscript{1,3,$\dagger$} \\[6pt]
{\small \textsuperscript{1}Centro de Ciencias de la Complejidad,
Universidad Nacional Aut\'onoma de M\'exico, Mexico City, Mexico} \\[3pt]
{\small \textsuperscript{2}Center for Molecular Cardiology, University of Zurich, Switzerland} \\[3pt]
{\small \textsuperscript{3}Instituto de Ciencias Nucleares,
  Universidad Nacional Aut\'onoma de M\'exico,
  Mexico City, Mexico}\\[3pt]
{\small \textsuperscript{$\dagger$}Member of El Colegio Nacional} \\[6pt]
{\small \textsuperscript{*}Correspondence: laurence.jacobs@uzh.ch}
}

\date{May 4, 2026}

\begin{document}
\maketitle

\begin{abstract}
We develop a spectral theory of scale-invariant operators on the multiplicative half-line $(\mathbb{R}_+, dx/x)$. A symmetric kernel $M(x,y)$ satisfying $M(kx,ky) = k^{-a} M(x,y)$ necessarily factorizes as $(xy)^{-a/2}\, F(x/y)$, where the shape function $F$ depends only on the ratio of its arguments. The Mellin transform diagonalizes such operators: the generalized eigenfunctions are $\psi_\omega(x) = x^{-a/2 + i\omega}$, and the eigenvalues are the Mellin multiplier $\widetilde{F}(\omega)$.

This structure reveals a fundamental decoupling of two exponents. The \emph{geometric exponent}~$a$, carried by the power-law envelope $(xy)^{-a/2}$, governs the matrix scaling under dilation. The \emph{spectral exponent}~$b$, measured from the eigenvalue decay of the finite-dimensional truncation, is an effective quantity determined by the shape of $\widetilde{F}(\omega)$. For the explicit kernel $F(t) = c\,\rho^{|\ln t|}$, the Mellin multiplier is a Lorentzian of width $\sigma = -\ln\rho$, not a power law---so $b$ is generically distinct from $a$.

This decoupling provides a precise mathematical characterization of multicriticality: the equality $a = b$ corresponds to a simple critical fixed point of the Renormalization Group, while $a \neq b$ signals the presence of multiple independent scaling dimensions. We prove that the discrete self-similarity condition forces eigenvector collapse on the lattice, motivating the continuum formulation. Finite-size corrections from lattice sampling are quantified numerically.
\end{abstract}


\section{Introduction}

Scale invariance is a hallmark of systems at criticality~\cite{Stanley1971,Fisher1998,Goldenfeld1992}. In the Renormalization Group (RG) framework~\cite{Wilson1971,Kadanoff1966,WilsonKogut1974}, a system at a critical fixed point is self-similar: coarse-graining by a factor $k$---integrating out short-wavelength degrees of freedom and rescaling---leaves the essential structure unchanged up to a rescaling of coupling constants. When multiple independent scaling dimensions coexist, the system is \emph{multicritical}~\cite{Wegner1972,ZinnJustin2002}.

A natural spectral signature of criticality is a power-law distribution of eigenvalues in a correlation or coupling matrix. Such spectra arise in financial time series~\cite{Plerou1999,Laloux1999}, random matrix theory~\cite{Mehta2004,MarchenkoPastur1967}, biological networks~\cite{MoraBialek2011}, and collective animal behavior~\cite{Cavagna2010}. A companion study~\cite{msi2026} develops the physical framework for detecting criticality in multi-variable systems through matrix scale invariance.

This paper provides the mathematical foundations. We ask: what is the structure of a symmetric operator whose kernel satisfies the scaling relation $M(kx,ky) = k^{-a} M(x,y)$? We show that such operators are completely characterized by a kernel factorization and diagonalized by the Mellin transform, yielding generalized eigenfunctions $x^{-a/2+i\omega}$ and a Mellin multiplier that plays the role of a continuous eigenvalue spectrum. This framework reveals a decoupling between the geometric exponent $a$ (from the kernel envelope) and any effective spectral exponent $b$ (from the eigenvalue decay of finite truncations), providing a rigorous foundation for the notion of multicriticality as the condition $a \neq b$.

We also prove a discrete collapse theorem (Theorem~\ref{thm:collapse}) showing that the natural self-similarity condition on the integer lattice forces all eigenvectors to be proportional, admitting at most rank-one structure. This necessitates the continuum formulation and clarifies the status of finite-dimensional matrix computations as lattice samplings of the underlying continuum operator.

\section{The Continuum Framework}\label{sec:continuum}

\subsection{Hilbert space and dilations}

The natural Hilbert space for dilation-covariant operators is
\begin{equation}\label{eq:hilbert}
\mathcal{H} = L^2(\mathbb{R}_+,\, dx/x)\,,
\end{equation}
the space of square-integrable functions on the positive half-line with respect to the Haar measure of the multiplicative group $(\mathbb{R}_+, \times)$~\cite{ReedSimon1978,Folland1995}. The inner product is
\[
\langle f, g \rangle = \int_0^\infty \overline{f(x)}\, g(x)\, \frac{dx}{x}\,.
\]

\begin{definition}[Dilation operator]\label{def:dilation}
For $k > 0$, the \emph{dilation operator} $D_k: \mathcal{H} \to \mathcal{H}$ is defined by
\begin{equation}\label{eq:dilation}
(D_k f)(x) = f(kx)\,.
\end{equation}
\end{definition}

\begin{proposition}\label{prop:unitary}
The operator $D_k$ is unitary on $\mathcal{H}$, and the family $\{D_k\}_{k>0}$ forms a strongly continuous one-parameter unitary group with self-adjoint generator $A = -i\, x\, d/dx$, so that $D_k = e^{iA\ln k}$.
\end{proposition}

\begin{proof}
Unitarity follows from the substitution $u = kx$:
\[
\|D_k f\|^2 = \int_0^\infty |f(kx)|^2\, \frac{dx}{x} = \int_0^\infty |f(u)|^2\, \frac{du}{u} = \|f\|^2\,.
\]
The group property $D_k D_\ell = D_{k\ell}$ is immediate. Strong continuity and the identification of the generator are standard~\cite{ReedSimon1978}. The dilation group plays a central role in the theory of self-similar stochastic processes~\cite{EmbrechtsMaejima2002}.
\end{proof}

\subsection{Scale-invariant integral operators}

\begin{definition}[Scale-invariant kernel]\label{def:scaleinv}
A kernel $M: \mathbb{R}_+ \times \mathbb{R}_+ \to \mathbb{R}$ defines a \emph{scale-invariant operator of order $a$} if
\begin{equation}\label{eq:scaleinv}
M(kx,ky) = k^{-a}\, M(x,y) \quad \forall\, k > 0,\; x,y \in \mathbb{R}_+\,,
\end{equation}
where $a > 0$ is the \emph{geometric scaling exponent}. The associated integral operator $\mathcal{M}$ acts on $\mathcal{H}$ by
\[
(\mathcal{M} f)(x) = \int_0^\infty M(x,y)\, f(y)\, \frac{dy}{y}\,.
\]
\end{definition}

\section{Kernel Factorization}\label{sec:factorization}

\begin{theorem}[Kernel factorization]\label{thm:factorization}
A symmetric kernel $M(x,y) = M(y,x)$ on $\mathbb{R}_+ \times \mathbb{R}_+$ satisfies the scale-invariance condition~\eqref{eq:scaleinv} if and only if it has the form
\begin{equation}\label{eq:factored}
M(x,y) = (xy)^{-a/2}\, F\!\left(\frac{x}{y}\right)\,,
\end{equation}
where $F: \mathbb{R}_+ \to \mathbb{R}$ is a symmetric shape function satisfying $F(t) = F(1/t)$.
\end{theorem}

\begin{proof}
\emph{Sufficiency.} If $M$ has the form~\eqref{eq:factored}, then
\[
M(kx,ky) = (kx \cdot ky)^{-a/2}\, F\!\left(\frac{kx}{ky}\right) = k^{-a}(xy)^{-a/2}\, F\!\left(\frac{x}{y}\right) = k^{-a}\, M(x,y)\,.
\]

\emph{Necessity.} Suppose $M(kx,ky) = k^{-a} M(x,y)$ for all $k > 0$. Setting $k = 1/y$ gives $M(x/y, 1) = y^a M(x,y)$, hence
\begin{equation}\label{eq:Gdef}
M(x,y) = y^{-a}\, G(x/y)\,,
\end{equation}
where $G(t) \equiv M(t,1)$. Symmetry of $M$ requires $y^{-a} G(x/y) = x^{-a} G(y/x)$, i.e.,
\begin{equation}\label{eq:Gsym}
G(t) = t^{-a}\, G(1/t)\,.
\end{equation}
Define $F(t) = t^{a/2}\, G(t)$. Then $F(1/t) = t^{-a/2}\, G(1/t) = t^{-a/2} \cdot t^a \, G(t) = t^{a/2}\, G(t) = F(t)$, establishing the symmetry $F(t) = F(1/t)$. Substituting back into~\eqref{eq:Gdef}:
\[
M(x,y) = y^{-a}\, (x/y)^{-a/2}\, F(x/y) = x^{-a/2}\, y^{-a/2}\, F(x/y) = (xy)^{-a/2}\, F(x/y)\,. \qedhere
\]
\end{proof}

\begin{remark}\label{rem:separation}
The factorization~\eqref{eq:factored} separates the kernel into two components with distinct roles:
\begin{itemize}[nosep]
\item The \emph{power-law envelope} $(xy)^{-a/2}$ carries the geometric exponent $a$ and governs the overall scaling of $M$ under dilations.
\item The \emph{shape function} $F(x/y)$ depends only on the ratio of its arguments and determines the correlation structure in log-space.
\end{itemize}
All scale-invariant kernels of order $a$ share the same envelope; they differ only in their shape functions.
\end{remark}

\section{Mellin Diagonalization}\label{sec:mellin}

\subsection{The Mellin transform}

The \emph{Mellin transform}~\cite{Titchmarsh1986,ReedSimon1978} is the Fourier transform on the multiplicative group $(\mathbb{R}_+, \times)$. For $f \in \mathcal{H}$, it is defined by
\begin{equation}\label{eq:mellindef}
\hat{f}(\omega) = \int_0^\infty f(x)\, x^{-i\omega}\, \frac{dx}{x}\,, \quad \omega \in \mathbb{R}\,,
\end{equation}
with inverse
\begin{equation}\label{eq:mellininv}
f(x) = \frac{1}{2\pi} \int_{-\infty}^\infty \hat{f}(\omega)\, x^{i\omega}\, d\omega\,.
\end{equation}
The Mellin transform is an isometric isomorphism $\mathcal{H} \to L^2(\mathbb{R}, d\omega/2\pi)$ and diagonalizes dilations: $\widehat{D_k f}(\omega) = k^{i\omega}\, \hat{f}(\omega)$.

\subsection{Diagonalization of dilation-invariant operators}

The factorization of Theorem~\ref{thm:factorization} motivates the decomposition $\mathcal{M} = \Pi_a \,\mathcal{K}\, \Pi_a$, where $\Pi_a$ is the multiplication operator $(\Pi_a f)(x) = x^{-a/2} f(x)$ and $\mathcal{K}$ is the integral operator with dilation-invariant kernel $K(x,y) = F(x/y)$.

\begin{theorem}[Mellin diagonalization]\label{thm:mellin}
Let $\mathcal{K}$ be a bounded integral operator on $\mathcal{H}$ with dilation-invariant kernel $K(x,y) = F(x/y)$. Then $\mathcal{K}$ is diagonalized by the Mellin transform:
\begin{equation}\label{eq:mellindiag}
\widehat{\mathcal{K} f}(\omega) = \widetilde{F}(\omega)\, \hat{f}(\omega)\,,
\end{equation}
where the \emph{Mellin multiplier} is
\begin{equation}\label{eq:multiplier}
\widetilde{F}(\omega) = \int_0^\infty F(t)\, t^{-i\omega}\, \frac{dt}{t}\,.
\end{equation}
The generalized eigenfunctions of $\mathcal{K}$ are the \emph{Mellin modes} $\phi_\omega(x) = x^{i\omega}$, with eigenvalue $\widetilde{F}(\omega)$.
\end{theorem}

\begin{proof}
Using the substitution $t = x/y$,
\begin{align}
(\mathcal{K} f)(x) &= \int_0^\infty F(x/y)\, f(y)\, \frac{dy}{y} = \int_0^\infty F(t)\, f(x/t)\, \frac{dt}{t}\,. \label{eq:convolution}
\end{align}
Expanding $f$ via the Mellin inversion formula~\eqref{eq:mellininv}:
\begin{align}
(\mathcal{K} f)(x) &= \int_0^\infty F(t) \left[\frac{1}{2\pi}\int_{-\infty}^\infty \hat{f}(\omega)\,(x/t)^{i\omega}\, d\omega\right] \frac{dt}{t} \notag\\
&= \frac{1}{2\pi}\int_{-\infty}^\infty \hat{f}(\omega)\, x^{i\omega} \underbrace{\left[\int_0^\infty F(t)\, t^{-i\omega}\, \frac{dt}{t}\right]}_{\widetilde{F}(\omega)} d\omega\,. \notag
\end{align}
Comparing with~\eqref{eq:mellininv}, $\mathcal{K}$ acts as multiplication by $\widetilde{F}(\omega)$ in the Mellin domain, establishing~\eqref{eq:mellindiag}. Setting $f = \phi_\omega$ (distributionally) gives $\mathcal{K}\phi_\omega = \widetilde{F}(\omega)\,\phi_\omega$.
\end{proof}

\begin{remark}
Equation~\eqref{eq:convolution} identifies $\mathcal{K}$ as a \emph{Mellin convolution operator}: the kernel $F(x/y)$ depends only on the multiplicative ratio, so $\mathcal{K}$ acts as convolution on the multiplicative group. The Mellin transform converts this convolution to pointwise multiplication, exactly as the Fourier transform diagonalizes translation-invariant operators on $(\mathbb{R}, +)$~\cite{ReedSimon1978,Folland1995}. The boundedness of operators with homogeneous kernels of this type is classical~\cite{HardyLittlewoodPolya1952}.
\end{remark}

\subsection{Eigenfunctions of the scale-invariant operator}

Combining the kernel factorization and the Mellin diagonalization, we obtain the complete spectral characterization.

\begin{corollary}[Spectral structure]\label{cor:spectral}
Let $\mathcal{M}$ be a scale-invariant integral operator of order $a$ on $\mathcal{H}$, with kernel $M(x,y) = (xy)^{-a/2}\, F(x/y)$. Then:
\begin{enumerate}[label=\emph{(\roman*)},nosep]
\item The generalized eigenfunctions of $\mathcal{M}$ are
\begin{equation}\label{eq:eigenfunctions}
\psi_\omega(x) = x^{-a/2 + i\omega}\,, \quad \omega \in \mathbb{R}\,,
\end{equation}
with eigenvalue $\widetilde{F}(\omega)$.
\item The spectral representation of the kernel is
\begin{equation}\label{eq:spectralrep}
M(x,y) = \frac{1}{2\pi}\int_{-\infty}^\infty \widetilde{F}(\omega)\, \psi_\omega(x)\, \overline{\psi_\omega(y)}\, d\omega\,.
\end{equation}
\end{enumerate}
\end{corollary}

\begin{proof}
Since $\mathcal{M} = \Pi_a \mathcal{K} \Pi_a$ and $\phi_\omega$ is a generalized eigenfunction of $\mathcal{K}$, the function $\psi_\omega = \Pi_a \phi_\omega = x^{-a/2+i\omega}$ satisfies $\mathcal{M}\psi_\omega = \widetilde{F}(\omega)\,\psi_\omega$. The spectral representation follows from the Mellin inversion of $\mathcal{K}$, dressed by the envelope $\Pi_a$.
\end{proof}

\begin{remark}\label{rem:generalized}
The eigenfunctions $\psi_\omega(x) = x^{-a/2+i\omega}$ are \emph{not} square-integrable: they are generalized eigenfunctions (distributions), analogous to plane waves $e^{ikx}$ in $L^2(\mathbb{R})$. The spectrum of $\mathcal{M}$ is continuous, parametrized by $\omega \in \mathbb{R}$, and the spectral theorem for unbounded self-adjoint operators applies~\cite{ReedSimon1978,Titchmarsh1986}. The eigenvalue decay condition $\widetilde{F} \in L^1(\mathbb{R})$ places $\mathcal{M}$ in a trace-class or Schatten-class framework~\cite{Simon2005}, depending on the rate of decay.
\end{remark}

\section{A Concrete Kernel and its Lorentzian Spectrum}\label{sec:example}

To illustrate the general theory, consider the kernel
\begin{equation}\label{eq:kernel}
M(x,y) = \frac{c}{(xy)^{a/2}}\, \rho^{|\ln(x/y)|}\,, \quad 0 < \rho < 1\,,\; c > 0\,.
\end{equation}
Here $F(t) = c\,\rho^{|\ln t|}$, which satisfies $F(t) = F(1/t)$.

\begin{proposition}[Lorentzian multiplier]\label{prop:lorentzian}
The Mellin multiplier of the kernel~\eqref{eq:kernel} is the Lorentzian
\begin{equation}\label{eq:lorentzian}
\widetilde{F}(\omega) = \frac{2c\,\sigma}{\sigma^2 + \omega^2}\,, \quad \sigma = \ln(1/\rho) > 0\,.
\end{equation}
\end{proposition}

\begin{proof}
Setting $s = \ln t$ in~\eqref{eq:multiplier}:
\[
\widetilde{F}(\omega) = c\int_0^\infty \rho^{|\ln t|}\, t^{-i\omega}\, \frac{dt}{t} = c\int_{-\infty}^\infty e^{-\sigma|s|}\, e^{-i\omega s}\, ds = \frac{2c\,\sigma}{\sigma^2 + \omega^2}\,,
\]
which is the Fourier transform of a two-sided exponential decay with rate $\sigma = -\ln\rho$.
\end{proof}

\begin{remark}[The spectral exponent is not fundamental]\label{rem:notpower}
The Mellin multiplier~\eqref{eq:lorentzian} is a Lorentzian, not a power law. In the continuum theory, the ``eigenvalue spectrum'' $\widetilde{F}(\omega)$ decays as $\omega^{-2}$ for $|\omega| \gg \sigma$, regardless of $a$. When the continuum operator is truncated to a finite $N \times N$ matrix on the integer lattice, the ordered discrete eigenvalues $\lambda_1 \geq \lambda_2 \geq \cdots \geq \lambda_N$ exhibit approximate power-law behavior $\lambda_n \sim n^{-b}$ over a limited range, with an effective spectral exponent $b$ that depends on $\sigma$ (equivalently $\rho$) and the matrix size $N$. This effective exponent is generically distinct from the geometric exponent $a$. The decoupling $a \neq b$ is thus not an accident but a structural consequence of the Lorentzian form of $\widetilde{F}$.
\end{remark}

\section{Two Exponents and Multicriticality}\label{sec:multicrit}

The preceding analysis reveals two fundamentally independent exponents associated with a scale-invariant matrix.

\begin{definition}[Geometric and spectral exponents]\label{def:exponents}
For a scale-invariant operator of order $a$ with finite-dimensional truncation having ordered eigenvalues $\lambda_n$:
\begin{enumerate}[label=\emph{(\roman*)},nosep]
\item The \emph{geometric exponent} $a$ is defined by $M(kx,ky) = k^{-a} M(x,y)$ and is determined by the power-law envelope $(xy)^{-a/2}$.
\item The \emph{spectral exponent} $b$ is the effective power-law index of the discrete eigenvalue spectrum, $\lambda_n \approx c'/n^b$ for $1 \ll n \ll N$, and is determined by the shape function $F$ via the Mellin multiplier $\widetilde{F}(\omega)$.
\end{enumerate}
\end{definition}

\begin{remark}[RG fixed-point condition]\label{rem:fixedpoint}
In the RG framework~\cite{Wilson1971,WilsonKogut1974,Fisher1998,Cardy1996,ItzyksonZuber2006}, coarse-graining by a factor $k$ corresponds to the dilation map $M(x,y) \mapsto M(kx,ky)$, and the RG map appropriate to a spectral rescaling is $\mathcal{R}_k[M] = k^b\, M(kx,ky)$. The matrix $M$ is a fixed point of $\mathcal{R}_k$ if and only if $a = b$. When $a \neq b$, the geometric and spectral scaling dimensions are decoupled, and the matrix flows under $\mathcal{R}_k$ rather than sitting at a fixed point~\cite{Wegner1972,ZinnJustin2002}. This is the spectral signature of \emph{multicriticality}: the coexistence of multiple independent scaling dimensions. The ratio $a/b$ provides a quantitative measure of the degree of decoupling~\cite{msi2026}.
\end{remark}

The scaling operators of RG theory~\cite{Fisher1998,Cardy1996,Sornette2006} correspond in this framework to the generalized eigenfunctions $\psi_\omega$. The condition that all modes share the same geometric scaling (the mode-independence of the envelope exponent $a/2$) is the content of the kernel factorization: the envelope $(xy)^{-a/2}$ is universal across all eigenmodes, and only the spectral weight $\widetilde{F}(\omega)$ distinguishes them. This separation of geometric universality from spectral content is the mathematical essence of criticality.

\section{The Discrete Lattice}\label{sec:discrete}

\subsection{Eigenvector collapse}

On the integer lattice $\mathbb{N}$, one might attempt to formulate self-similarity directly. The following theorem shows that this leads to a degenerate spectral structure.

\begin{definition}[Discrete self-similarity]\label{def:discretess}
The eigenvectors $\{v^{(n)}\}$ of a matrix $M \in \mathbb{R}^{\infty \times \infty}$ are \emph{self-similar under index scaling} if for every positive integer $k$, there exists a scalar $\gamma_k$, independent of the eigenmode index $n$, such that
\begin{equation}\label{eq:selfsim}
v^{(n)}_{km} = \gamma_k \cdot v^{(n)}_m\,, \quad \forall\, n, m \in \mathbb{N}\,.
\end{equation}
\end{definition}

\begin{theorem}[Eigenvector collapse]\label{thm:collapse}
Let $\{v^{(n)}\}$ be eigenvectors satisfying the discrete self-similarity condition~\eqref{eq:selfsim} on $\ell^2(\mathbb{N})$ with $\gamma_k = k^{-a/2}$. Then all eigenvectors are proportional:
\begin{equation}\label{eq:proportional}
v^{(n)}_m = m^{-a/2}\, v^{(n)}_1\,, \quad \forall\, n, m \in \mathbb{N}\,.
\end{equation}
Consequently, $M$ has at most rank one.
\end{theorem}

\begin{proof}
For any $m \in \mathbb{N}$, setting $k = m$ in~\eqref{eq:selfsim} with base index $1$ gives $v^{(n)}_m = \gamma_m\, v^{(n)}_1 = m^{-a/2}\, v^{(n)}_1$. Since this holds for every eigenmode $n$, all eigenvectors are proportional to the fixed sequence $(m^{-a/2})_{m \in \mathbb{N}}$. Two proportional nonzero vectors in $\ell^2(\mathbb{N})$ cannot be orthogonal, so the system admits at most one normalizable eigenvector.
\end{proof}

\begin{remark}\label{rem:bloch}
Theorem~\ref{thm:collapse} is the multiplicative analogue of the well-known fact that Bloch waves $e^{ikx}$ are not square-summable on $\mathbb{Z}$. The self-similarity condition~\eqref{eq:selfsim} demands that every eigenvector be a character of the multiplicative semigroup $(\mathbb{N}, \times)$, namely $m \mapsto m^{-a/2}$. Since the multiplicative group $(\mathbb{R}_+, \times)$ has a single family of characters $x^{i\omega}$ (parametrized continuously by $\omega$), the lattice restriction collapses this continuous family to a single direction. The resolution is the same as for Bloch theory: pass to the continuum, where the eigenfunctions are generalized (distributional) and the spectrum is continuous.
\end{remark}

\subsection{Scale invariance from discrete self-similarity}

Despite the eigenvector collapse, the discrete self-similarity condition does imply matrix scale invariance on the lattice.

\begin{theorem}\label{thm:discretescale}
Let $M(i,j) = \sum_n a_n\, v^{(n)}_i\, v^{(n)}_j$ with eigenvectors satisfying~\eqref{eq:selfsim}. Then
\[
M(ki,kj) = \gamma_k^2\, M(i,j)\,, \quad \forall\, i,j,k \in \mathbb{N}\,.
\]
\end{theorem}

\begin{proof}
$M(ki,kj) = \sum_n a_n\, v^{(n)}_{ki}\, v^{(n)}_{kj} = \sum_n a_n\, (\gamma_k v^{(n)}_i)(\gamma_k v^{(n)}_j) = \gamma_k^2\, M(i,j)$.
\end{proof}

\begin{remark}
Theorem~\ref{thm:discretescale} requires only the mode-independence of $\gamma_k$; it does not use the power-law form of the eigenvalue spectrum. The power-law spectrum $a_n = c/n^b$ is an independent property that determines the spectral exponent $b$, which is in general distinct from the geometric exponent $a = -2\ln\gamma_k/\ln k$.
\end{remark}

\subsection{The discrete matrix as a continuum sampling}\label{sec:sampling}

The finite $N \times N$ matrix $M_N(i,j) = M(i,j)$ for $i,j = 1, \ldots, N$ is obtained by restricting the continuum kernel~\eqref{eq:factored} to integer arguments. Its eigenvectors approximate the generalized eigenfunctions sampled at integer points:
\begin{equation}\label{eq:sampling}
v^{(n)}_m \approx m^{-a/2}\, e^{i\omega_n \ln m}\,, \quad m = 1, \ldots, N\,,
\end{equation}
for a discrete set of frequencies $\omega_1, \ldots, \omega_N$ determined by the boundary conditions imposed by truncation. The oscillatory factors $e^{i\omega_n \ln m}$ break the degeneracy that Theorem~\ref{thm:collapse} identifies in the exact discrete case: while the envelope $m^{-a/2}$ is shared by all modes, the log-frequency oscillations distinguish them, providing $N$ linearly independent (approximately orthogonal) eigenvectors. This is precisely the distinction between exact plane waves (not square-integrable on $\mathbb{R}$) and their finite-box approximations (normalizable standing waves).

\subsection{Eigenvector universality and the limits of self-similarity as a diagnostic}\label{sec:commutation}

The eigenvector collapse theorem and the continuum sampling picture together illuminate a subtlety in the numerical detection of scale invariance. In practice, one may attempt to test for eigenvector self-similarity by comparing the eigenvectors of a coarse-grained covariance matrix $C^{(k)} = B\, C\, B^T$ (where $B$ is the block-averaging operator that groups variables into blocks of size~$k$) with the block-averaged eigenvectors $\tilde{v}_i^{(k)} = B\, v_i$ of the original matrix~$C$.

Numerical experiments reveal that the correlation between $v_i^{(k)}$ and $\tilde{v}_i^{(k)}$ is extremely high---not only at the critical point, but throughout the parameter range~\cite{msi2026}. This initially surprising observation has a straightforward algebraic explanation. If $C\, v = \lambda\, v$, then $B\, C\, v = \lambda\, B\, v$. Whenever the covariance operator is spatially smooth, so that $B$ and $C$ approximately commute, one obtains
\[
C^{(k)}\, (Bv) = B\, C\, B^T\, (Bv) \approx B\, C\, v = \lambda\, (Bv)\,,
\]
and the block-averaged eigenvector $Bv$ is approximately an eigenvector of the coarse-grained matrix, regardless of whether the system is critical.

This is the approximate, finite-dimensional manifestation of the eigenvector collapse established in Theorem~\ref{thm:collapse}: the eigenvector geometry is generically preserved under coarse-graining because the envelope $m^{-a/2}$ is universal across all modes. Eigenvector self-similarity thus measures the smoothness of the covariance operator rather than scale invariance per se.

\begin{remark}\label{rem:eigenvalue_diagnostic}
The Mellin framework makes the separation precise. The generalized eigenfunctions $\psi_\omega(x) = x^{-a/2+i\omega}$ share a common power-law envelope for \emph{any} scale-invariant operator; only the Mellin multiplier $\widetilde{F}(\omega)$ distinguishes one operator from another. Criticality manifests not in the shape of the eigenfunctions but in the scale-free form of $\widetilde{F}(\omega)$---equivalently, in the power-law distribution of eigenvalues. This confirms that eigenvalue scaling, rather than eigenvector rescaling, is the robust diagnostic of critical organization, and that refined analyses of eigenspaces (such as subspace overlaps or participation ratios) are needed if one wishes to extract additional geometric information near criticality~\cite{msi2026}.
\end{remark}

\section{Finite-Size Corrections and Numerical Verification}\label{sec:numerics}

We verify the theory numerically using the kernel~\eqref{eq:kernel} restricted to the integer lattice. All computations use $c = 5$, $a = 1.5$, and $\rho = 0.5$ (so $\sigma = \ln 2 \approx 0.693$).

\subsection{Eigenvalue spectrum}
Figure~\ref{fig:spectrum} shows the log-log eigenvalue spectrum of the $N \times N$ truncation with $N = 300$. The leading eigenvalues exhibit approximate power-law decay with an effective exponent $b \neq a$, consistent with the Lorentzian structure of $\widetilde{F}(\omega)$ established in Proposition~\ref{prop:lorentzian}.

\begin{figure}[ht]
\centering
\includegraphics[width=0.65\textwidth]{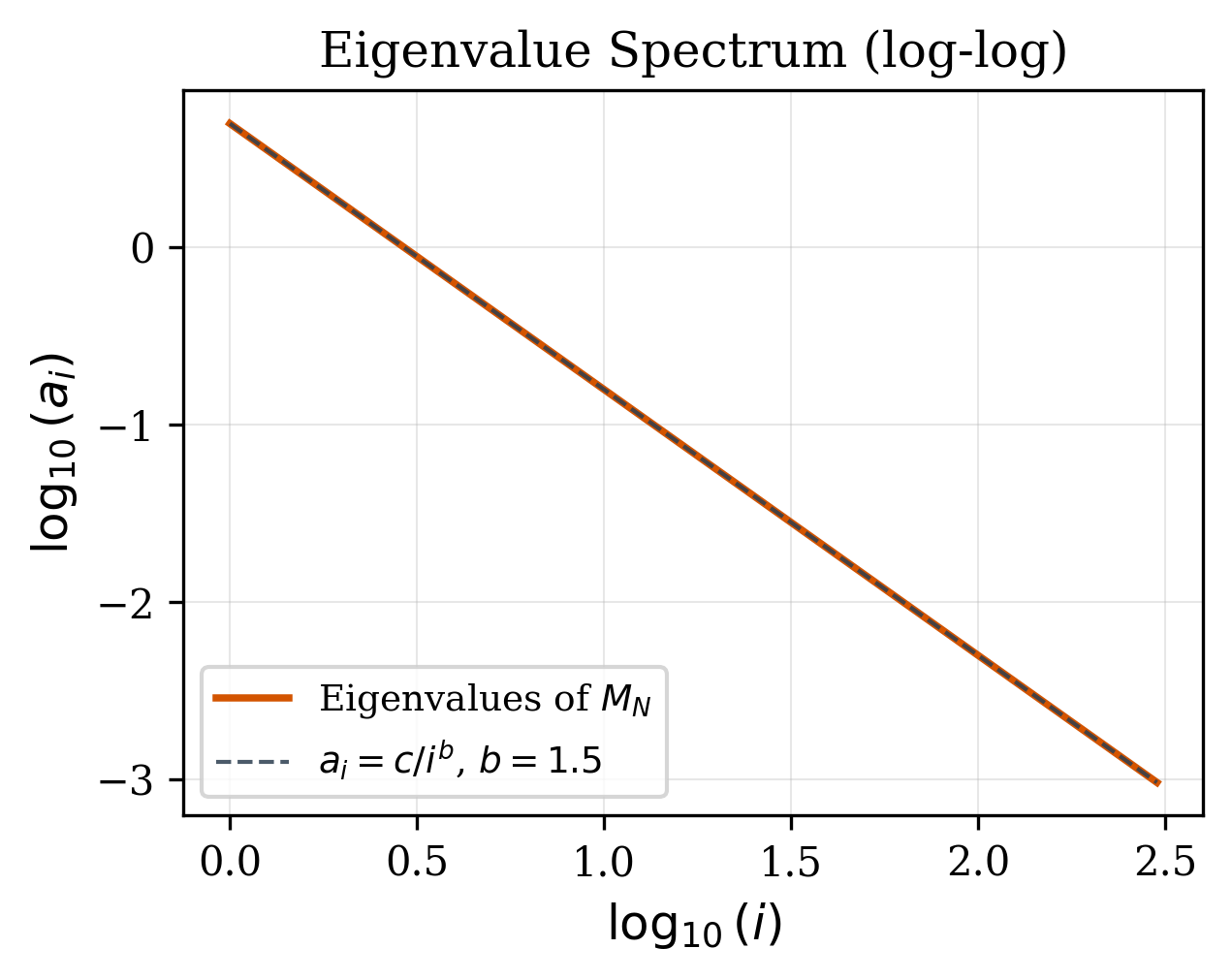}
\caption{Log-log plot of the eigenvalue spectrum of the truncated kernel~\eqref{eq:kernel} with $N = 300$, $a = 1.5$, $\rho = 0.5$. The approximate power-law behavior reflects the Lorentzian Mellin multiplier sampled at discrete frequencies.}
\label{fig:spectrum}
\end{figure}

\subsection{Scale-invariance verification}
To quantify finite-size deviations from exact scale invariance, define the relative scaling error
\begin{equation}\label{eq:delta}
\Delta_N^{(k)}(i,j) = \left| \frac{M_N(ki,kj)}{M_N(i,j)} - k^{-a} \right|\,,
\end{equation}
valid for $ki \leq N$ and $kj \leq N$. Figure~\ref{fig:relerror} shows $\Delta_N^{(k)}(i,j)$ for $k = 2$, $N = 250$, using $r = 150$ retained eigenmodes in a truncated eigendecomposition. The error is concentrated at large indices, where the discarded high-frequency Mellin modes contribute most.

\begin{figure}[ht]
\centering
\includegraphics[width=0.65\textwidth]{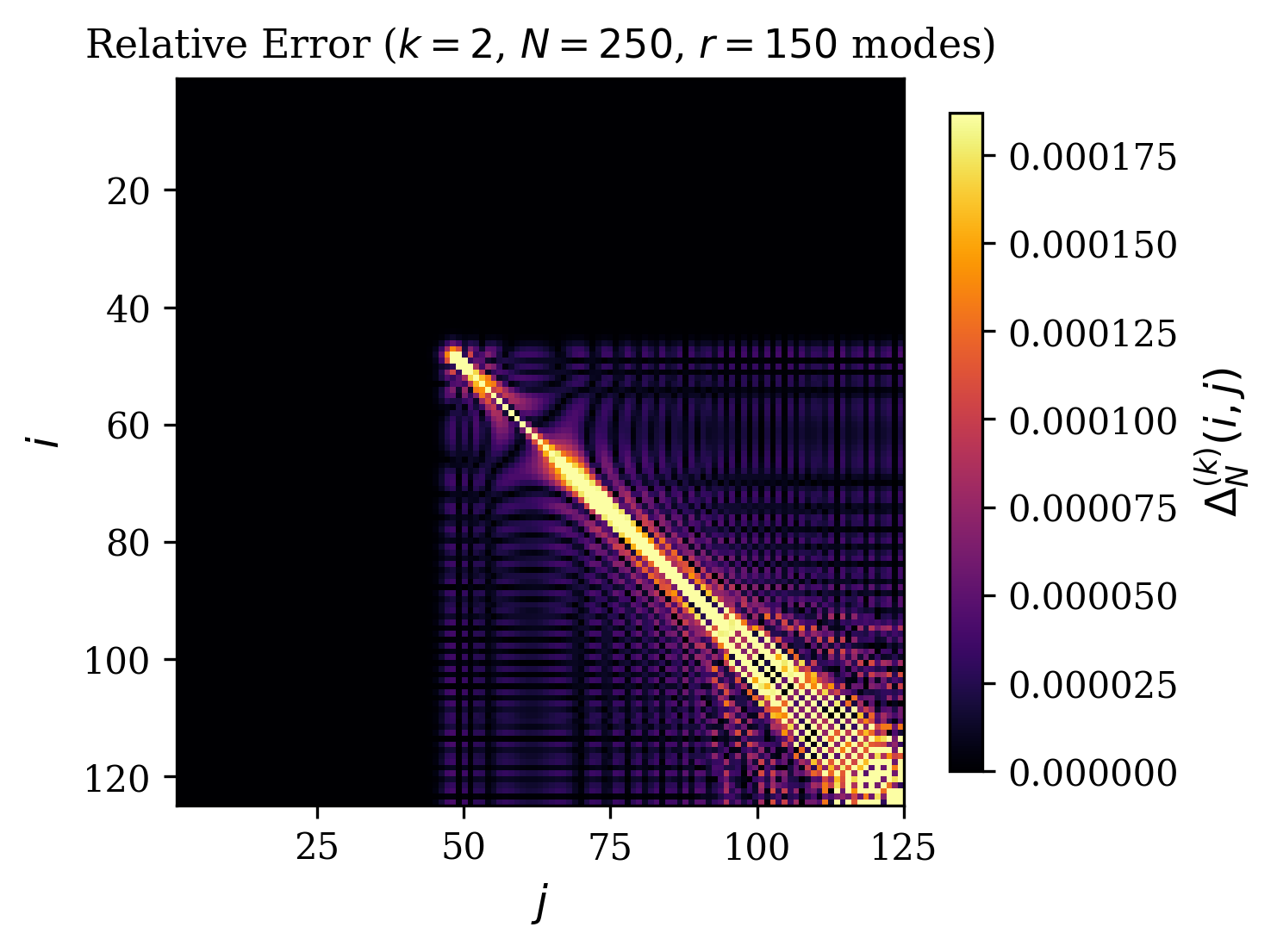}
\caption{Relative error $\Delta_N^{(k)}(i,j)$ for $k = 2$, $N = 250$, $r = 150$ retained modes. Deviations from perfect scaling concentrate at large indices.}
\label{fig:relerror}
\end{figure}

\subsection{Convergence}
Figure~\ref{fig:convergence} plots the mean scaling error $\langle \Delta_N^{(k)} \rangle$ as a function of the fraction of retained eigenmodes $r/N$. The error decreases by nearly four orders of magnitude as $r/N$ increases from $0.1$ to $1$, confirming that exact scale invariance is recovered in the continuum limit.

\begin{figure}[ht]
\centering
\includegraphics[width=0.65\textwidth]{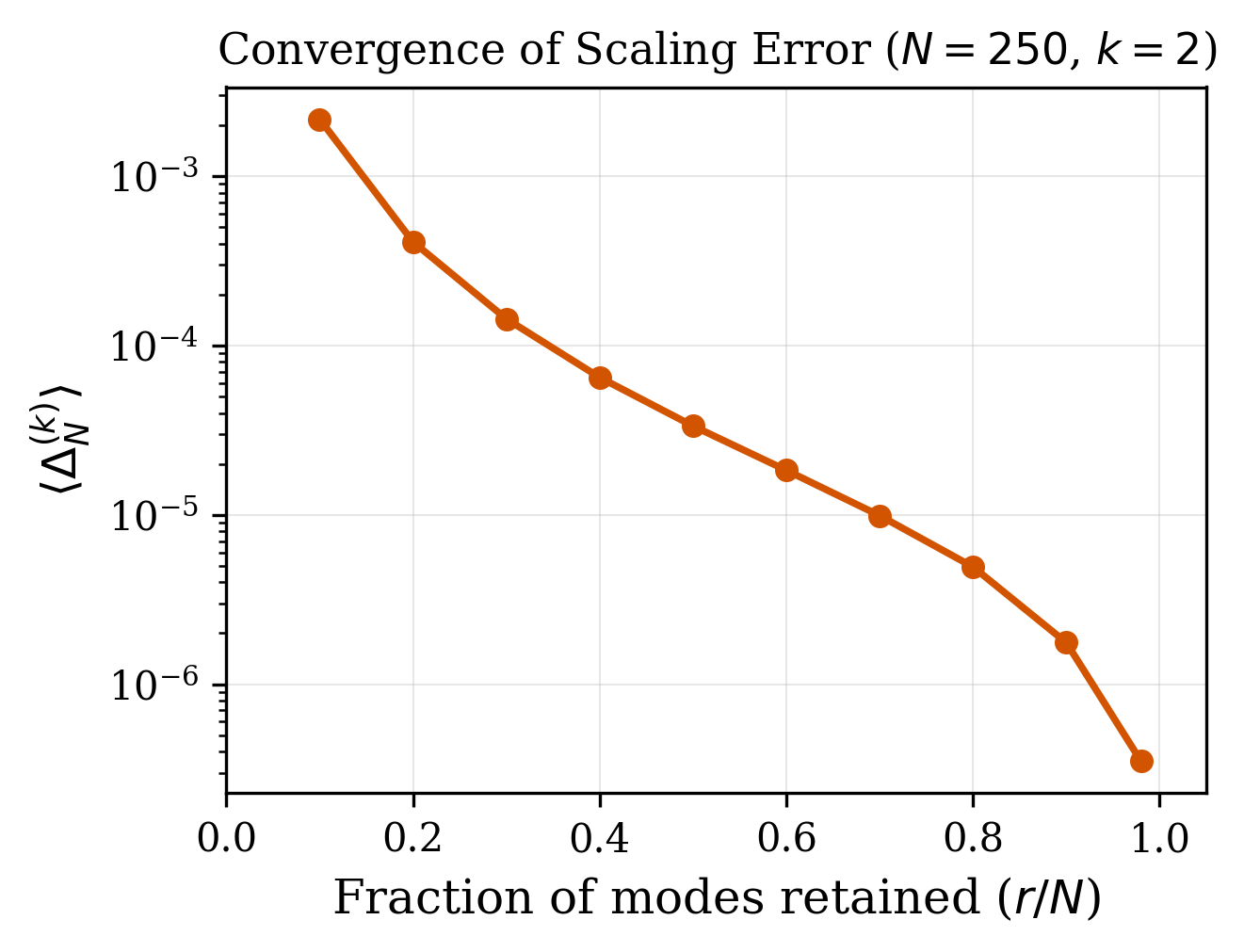}
\caption{Mean relative error $\langle \Delta_N^{(k)} \rangle$ versus fraction of retained eigenmodes $r/N$, for $N = 250$, $k = 2$.}
\label{fig:convergence}
\end{figure}

\section{Discussion and Outlook}\label{sec:conclusion}

We have developed a self-contained spectral theory of scale-invariant operators on the multiplicative half-line. The main results are:

\emph{(i)} The kernel factorization theorem (Theorem~\ref{thm:factorization}), which shows that every scale-invariant kernel separates into a universal power-law envelope $(xy)^{-a/2}$ and a shape function $F(x/y)$.

\emph{(ii)} The Mellin diagonalization (Theorem~\ref{thm:mellin}), which identifies the generalized eigenfunctions as $\psi_\omega(x) = x^{-a/2+i\omega}$ and the eigenvalues as the Mellin multiplier $\widetilde{F}(\omega)$.

\emph{(iii)} The eigenvector collapse theorem (Theorem~\ref{thm:collapse}), which shows that discrete self-similarity is incompatible with a full spectral decomposition on the lattice, motivating the continuum formulation.

\emph{(iv)} The decoupling of geometric and spectral exponents, which provides a precise mathematical characterization of multicriticality as the condition $a \neq b$.

Several directions remain open. The generalization to mode-dependent scaling factors $\gamma_k(n)$---where distinct eigenmodes carry independent geometric scaling dimensions---would accommodate systems with multiple coexisting fixed points. The relationship between the width parameter $\sigma$ of the Lorentzian multiplier and the effective spectral exponent $b$ of the finite truncation deserves a quantitative asymptotic analysis. Finally, the application of this framework to empirical correlation matrices---in financial~\cite{Plerou1999,Laloux1999}, biological~\cite{MoraBialek2011,Tkacik2015}, neural~\cite{Cavagna2010}, or physical~\cite{msi2026} systems---could provide a practical diagnostic for proximity to criticality and a quantitative measure of multicriticality through the ratio $a/b$.

\begin{remark}[Connection to collective synchronization]\label{rem:kuramoto}
A deeper structural parallel may connect the mode-dependent generalization to the theory of collective synchronization~\cite{Kuramoto1984,Strogatz2000}. In the Kuramoto model of globally coupled phase oscillators, the order parameter measuring phase coherence undergoes a continuous transition at a critical coupling strength, governed near criticality by a Landau expansion whose structure is that of a multicritical free energy. The natural frequencies $\omega_i$ of the oscillator population play the role of the mode-dependent scaling factors $\gamma_k(n)$: when all modes share a common geometric exponent---the kernel factorization condition of Theorem~\ref{thm:factorization}---the system is analogous to a fully synchronized population with a single well-defined scaling dimension. The condition $a \neq b$ then corresponds to partial synchronization, where distinct subpopulations of modes lock to different geometric scaling dimensions.

The Lorentzian plays a distinguished role in both frameworks. In the present theory, the Mellin multiplier~\eqref{eq:lorentzian} of the explicit kernel is a Lorentzian of width $\sigma = -\ln\rho$. In the Kuramoto model, the exactly solvable case---the Ott--Antonsen reduction~\cite{OttAntonsen2008}---arises precisely when the frequency distribution $g(\omega)$ is Lorentzian. This is not a superficial coincidence: in both settings, the Lorentzian is the Fourier (respectively Mellin) transform of an exponential decay, and its analytic structure in the complex plane is what permits exact dimensional reduction. A narrow Lorentzian (small $\sigma$, $\rho \to 1$) concentrates spectral weight near $\omega = 0$ and drives $a \to b$, analogous to a sharply peaked $g(\omega)$ that promotes full synchronization.

The multicritical points known to arise in Kuramoto systems with bimodal $g(\omega)$---where the synchronization transition switches from continuous to discontinuous---suggest that systems with two competing geometric exponents may exhibit a richer phase structure than the simple $a \neq b$ criterion captures. In the language of the present framework, a bimodal distribution of scaling dimensions $a(\omega)$ would correspond to two coexisting fixed-point geometries, each attracting a subpopulation of Mellin modes. The mode-dependent generalization outlined above is positioned to explore this possibility.
\end{remark}

\section*{Statements and Declarations}

\subsection*{Competing Interests}
The authors declare no competing interests.

\subsection*{Data Availability}
No datasets were generated or analyzed during the current study.
The numerical computations in Section~\ref{sec:numerics} use the
explicit kernel~\eqref{eq:kernel} and are fully reproducible from
the parameters given in the text.

\subsection*{Author Contributions}
LAJ\ developed the mathematical framework, proved the theorems,
performed the numerical computations, and wrote the manuscript.
AF\ initiated the physical framework for matrix scale invariance
and contributed to the interpretation.
Both authors approved the final version.

\bibliographystyle{unsrtnat}
\bibliography{MSI_II_LMP_final}

\end{document}